\documentclass[11pt]{article}
\usepackage{amsmath,amsthm,amsfonts,latexsym,amssymb,enumerate,color}
\usepackage{graphicx}
 \usepackage{hyperref}
 \usepackage{color}

\newtheorem{thm}{Theorem}[section]
\newtheorem{prop}[thm]{Proposition}
\newtheorem{lem}[thm]{Lemma}
\newtheorem{cor}[thm]{Corollary}
\newtheorem{remark}[thm]{Remark}   
\newtheorem{exam}[thm]{Example}

\def\DD{\mathbb D}
\def\TT{\mathbb T}

\def\dist{\mathop{\rm dist}\nolimits}
\def\imag{\mathop{\rm Im}\nolimits}
\def\re{\mathop{\rm Re}\nolimits}
\def\supp{\mathop{\rm supp}\nolimits}
\def\tr{\mathop{\rm tr}\nolimits}
\def\rank{\mathop{\rm rank}\nolimits}

\title{Norms of truncated Toeplitz operators and numerical radii of restricted shifts}

\author{Pamela Gorkin 
 and Jonathan R. Partington}

\begin{document}

\maketitle

\begin{abstract}
This paper gives a new approach to the calculation of the numerical radius of a restricted
shift operator by linking it to the norm of a truncated Toeplitz operator (TTO), which can 
be calculated by various methods. Further results on the norm of a TTO
are derived, and a conjecture on the existence of continuous symbols for compact
TTO is resolved.
\end{abstract}



\section{Introduction}

Given a bounded operator $T$ on a Hilbert space the {\it numerical range} of $T$ is defined by
\[W(T) = \{\langle Tx, x\rangle: \|x\| = 1\}.\]
For operators on spaces of dimension $2$ the elliptical range theorem \cite{GR, CKLI} tells us that the numerical range of $T$ is elliptical with foci at the eigenvalues $a,b$ of $T$ and minor axis of length 
\[\left(\tr(T^\star T) - |a|^2 - |b|^2\right)^{1/2}.\] For operators on spaces of larger dimension, it is usually difficult to describe the numerical range explicitly, though it can be done in many special cases. 
The numerical range plays an important role in applications 
such as the stability of linear systems (see, e.g., \cite{AM,fan-tits}),
 in part because it always contains the eigenvalues of the operator $T$. Even computing the {\it numerical radius} of $T$, denoted $w(T)$ and defined by
\[w(T) = \sup\{|\omega|: \omega \in W(T)\},\] is difficult but useful; for example, it gives us bounds on the norm of the operator:
\[\frac{1}{2} \|A\| \le w(A) \le \|A\|,\]
and more general estimates for analytic functions of the operator.   The authors of this paper were motivated to study the numerical radius partly because of its possible connections to the following:\\  

{\bf Crouzeix conjecture.} For every polynomial $p$ and matrix $A$, we have
\[\|p(A)\| \le 2 \max\{|p(z)|: z \in W(A)\}.\] 
This question was first stated in 2004 \cite{C2004} and important recent progress has been made \cite{CP2017}, but at the time of this writing the question remains open.  

In this paper we focus on a particular class of operators known as compressions of the shift operator and consider the numerical radius of the operators in this class, where little is known. Much more is known about the geometry of the numerical range of these operators and we refer the reader to
\cite{G2, GW, GW1, GW2, GW2016, M}.   Though the work in these papers may also be used to solve some of the problems mentioned here, our goal is to add other methods that may be useful. 

We turn to the class of operators of interest. Let $L^2$ denote the Lebesgue space on the circle $\TT=\partial\DD$, $H^2$ denote the classical Hardy space on the open unit disk $\mathbb{D}$, and $S : H^2 \to H^2$ defined by $(Sf)(z) = z  f(z)$ the {\it shift operator}. Let $P_-$ denote the orthogonal projection of $L^2$ onto $L^2 \ominus H^2$. An inner function is a bounded analytic function with radial limits of modulus $1$ almost everywhere. Beurling's theorem tells us that the closed, nontrivial, invariant subspaces for $S$ are precisely those of the form $u H^2$, with $u$ a nonconstant inner function. Thus, the nontrivial invariant subspaces for the adjoint of $S$, denoted $S^*$, are precisely the ones of the form $K_u = H^2 \ominus u H^2$. In this paper, we consider {\it compressions of the shift operator}: Let $u$ be an inner function and $S_u : K_u \to K_u$ be defined by $S_u = P_{K_u}S|_{K_u}$, where $P_{K_u}$ is the projection of $L^2$ onto $K_u$. On $H^2$ we have \[P_{K_u}(f) = f - uP_+(\overline{u}f) = u P_-(\overline{u} f).\] 

If the symbol $u$ of the compression of the shift is a discontinuous inner function, then $w(S_u) = 1$. Thus, as we study compressions of the shift operator, we focus on continuous inner functions; that is, we consider finite Blaschke products as the symbol when we try to determine the numerical radius of $S_u$.

The first main result in this paper allows us to compute the numerical radius of $S_B$ when $B$ is a finite Blaschke product and all zeros are real. This gives precise formulas for $w(S_B)$ for Blaschke products of degree up to $4$ and an algorithm for computing $w(S_B)$ when $B$ has degree greater than $4$. The proof is accomplished by using an algorithm developed by Foias and Tannenbaum (\cite{FT87}, see also \cite{FTZ88} and Section~\ref{sec:norm}) together with a result that connects the numerical range to the norms of certain closely related truncated Toeplitz operators. Then, via interpolation, we are able to produce an algorithm that involves computing the zeros of a special class of polynomials. In theory, then, the numerical radius is obtained via these zeros.

Note that every finite matrix $A$ that defines a contraction with spectrum in the unit disk satisfying $\rank(I-A^*A)=1$
is unitarily equvalent to an operator $S_B$ for $B$ a finite Blaschke product (see \cite{GW2}), so that our results apply in 
(formally) more general situations.

In Section~\ref{TTO}, we discuss more general {\it truncated Toeplitz operators} with symbol
$\varphi \in L^2$, or $A_\varphi^u$, defined on the dense subset $K_u \cap L^\infty$ of $K_u$ by 
\[A_\varphi^u(f) = P_{K_u}(\varphi f), ~\mbox{where}~ u~\mbox{is an inner function}.\]  
Work in the earlier sections focuses on the study of the distance from $(1 + a z)$ to $BH^\infty$ where $B$ is a finite Blaschke product, while the results of later sections consider the distance to the Sarason algebra, $H^\infty + C(\mathbb{T})$; that is, we consider  $\|g + B(H^\infty + C(\mathbb{T}))\|$ where $g \in H^\infty + C(\mathbb{T})$ and $B$ is an interpolating Blaschke product.  We obtain asymptotic distance estimates for a function $f \in H^\infty + C$ to $B(H^\infty + C(\mathbb{T})$ when $B$ is an interpolating Blaschke product, with more precise results in the case when the Blaschke product is a thin interpolating Blaschke product. Then we generalize a result of Ahern and Clark on compactness of truncated Toeplitz operators and conclude the paper with an example that provides an answer to a question stated in \cite{CFT}, constructing a compact truncated Toeplitz operator with no continuous symbol.

\section{Notation}
\label{Notation}

For $B$ a finite Blaschke product with zeros $a_j \in \mathbb{D}$, the corresponding reproducing kernels 
\[k_{a_j}(z) = \frac{1}{1 - \overline{a_j} z}\] lie in $K_B$, and in fact form a basis for $K_B$ when the zeros 
of $B$ are distinct.
When the zeros   are distinct, we obtain an orthonormal basis by applying the Gram--Schmidt process   to the reproducing kernels. This is called the {\it Takenaka--Malmquist basis}. In order for the matrix to be upper triangular, we find the matrix representation with respect to this basis in the reverse order and we note that we also will need to reorder the zeros of the Blaschke product. This can also be adjusted for non-distinct zeros and does not affect the matrix representation we give. In any event, the matrix representing the compression $S_B$ when $B$ is a finite Blaschke product with zeros $(a_j)$ can be written as
\begin{equation}
\label{eqn:A}
 a_{ij} = \left\{ \begin{array}{ll}
a_j & \mbox{if}~ i = j,\\
\big(\prod_{k = i + 1}^{j - 1}(-\overline{a_k})\big)\sqrt{1 -
|a_i|^2}\sqrt{1 -
|a_j|^2} & \mbox{if}~ i < j,\\ 0 & \mbox{if}~ i > j.
\end{array} \right. \end{equation}
For $\varphi \in L^\infty$, let $M_\varphi: L^2 \to L^2$ denote the {\it multiplication operator} defined by $M_\varphi f = \varphi f$. The {\it Toeplitz operator} $T_\varphi: H^2 \to H^2$ is defined by $T_\varphi = P_+ M_\varphi$ and the {\it Hankel operator} $H_\varphi: H^2 \to L^2 \ominus H^2$ is $H_\varphi = P_- M_\varphi$. It should be noted that $H_\varphi = H_\psi$ if and only if $\varphi - \psi \in H^\infty$. The Hankel operator $H_\varphi$ is compact if and only if $\varphi$ is in the algebra 
\[H^\infty + C(\mathbb{T})  : = \{f + c: f \in H^\infty, c \in C(\mathbb{T})\},\] where $C(\mathbb{T})$ is the algebra of continuous functions on the unit circle. This algebra is often called the {\it Sarason algebra}, in honor of D. Sarason who proved that it is a closed subalgebra of $L^\infty$. 

By considering operators defined on dense subsets, it is possible to  consider Toeplitz, Hankel and multiplication operators defined on $L^2$. In this spirit,
for $\varphi \in L^2$ we may define the truncated Toeplitz operator $A_\varphi^u$ on the dense subset $K_u \cap L^\infty$ of $K_u$ by 
\[A_\varphi^u(f) = P_{K_u}(\varphi f).\]  Recent surveys on truncated Toeplitz operators can be found in \cite{CFT, GR2013} and related results appear in \cite{B2015,CP16}.

\section{The norm computation}\label{sec:norm}

\subsection{Links with Hankel operators and interpolation}

Theorem~\ref{thm:main} below can be used to compute the numerical radius $w(T)$ of a bounded operator $T$. If we   maximize $\re W(e^{i \theta} T)$ over all $\theta$ with $0 \le \theta < 2\pi$, we will obtain $w(T)$  (see also \cite{KIPP}).

\begin{thm}\label{thm:main}\cite[Thm. 5, p. 17]{BD71} Let $T$ be an operator. Then

\[
\max \{ \re \lambda: \lambda \in W(T) \} = \lim_{a \to 0+} \frac{1}{a} \left\{
\|I+aT\|-1 \right\}.
\]
\end{thm}

 We are now able to rephrase the numerical radius computation in terms of a computation of the norm of an operator, which can be handled by the algorithm in \cite{FT87} that we describe in Section~\ref{sec:FT}.
 
  We first consider a Blaschke product with distinct zeros.   Let $a > 0$. We want $\|I+a S_B\| = ~\mbox{dist}((1 + a z)/B(z), H^\infty)$, the norm of a truncated Toeplitz operator, which is also well known to be  
the norm of the Hankel operator with symbol $(1+az)/B(z)$ as in \cite[Prop. 2.1]{sarason67}.

Similarly, for a general complex $t$ we may consider the norm of the Hankel operator with symbol $(1+tz)/B(z)$. Since a finite Blaschke product is continuous, $H_B$ is compact when $B$ is a finite Blaschke product.
When the norm is attained and takes the value $\gamma$ we have
\[
1+tz= B(z) g(z) + \gamma h(z),
\]
where $h \in H^\infty$ and $\|h\|_\infty \le 1$ (in fact it is a Blaschke product).
Equivalently, we can solve the interpolation
$h(z_k)=(1+tz_k)/\gamma$
for each $k$ where $h \in H^\infty$ with $\|h\|_\infty \le 1$ and the $z_k$ are the zeros of $B$.

By Nevanlinna--Pick theory (see, for example \cite{P97}) this is possible if and only if
the matrix $M$ with $(j,k)$ entry
\begin{equation}
  \frac{1- (1+t z_j)(1+ \overline{ t z_k})/\gamma^2 }{1-z_j \overline z_k } \label{eq:pickmatrix} 
\end{equation}
is positive semi-definite.
In the case that the $z_k$ are real, writing $t=|t|e^{i\theta}$, this becomes
\[
 \frac{1-(1+|t|^2 z_j z_k + |t| (z_j+z_k)\cos\theta+i |t|  (z_j-z_k)\sin\theta)/\gamma^2}{1-z_j z_k} .
\]

\subsection{The algorithm of Foias and Tannenbaum}\label{sec:FT}

We now recall the algorithm from \cite{FT87} and \cite{FTZ88}
as applied to the situation described above. By restricting to the special case that gives us the numerical radius,
we may make some simplifications, as follows. For the convenience of the reader
we shall mostly use the same notation as in \cite{FT87}.

For $\rho>0$, let 
\[
P_\rho = I - \frac{1}{4 \rho^2} (I+aS_B)(I+ \overline a S_B^*)
\]
and let
\[
\mu_*(z) = 1- B(z) \overline{B(0)}.
\]
The basic idea of \cite{FT87} is that $P_\rho$ will become singular for various values of $\rho$ and the largest $\rho$ is the norm of $(I + a S_B)/2$. Combining this with Theorem~\ref{thm:main} we will be able to provide the estimates necessary to compute the numerical radius of $S_B$ where $B$ is a finite Blaschke product. We introduce the relevant notation and then recall the main theorem of \cite{FT87} here for convenience.

Recalling the notation $(x \otimes y)(w) := \langle w, y\rangle x$ for $x, y,$ and $w$ in a Hilbert space $H$, we note that $I-S_BS_B^* = \mu_* \otimes \mu_*$ on $K_B$, see \cite{NF}. \\

Write $\nu_*(\rho)= P_\rho^{-1}\mu_*$ and
\[\nu(\rho) := \mbox{trace}(I - S_B S_B^\star) P_\rho^{-1} = \langle \nu_*(\rho),\mu_* \rangle.\]

The following is the main result of \cite{FT87}.

\begin{thm} Let $I$ be an inner function, $R$ a rational function, and let $\rho_s = \max\{|R(z)|: z \in \sigma(S_I)~\mbox{with}~|z| = 1\}$ ($\rho_s$ is the essential spectral radius of $R(S_I)$). Consider $\nu(\rho)$ as a function of $\rho$ in the interval $J: = (\rho_s, 1]$. If $\nu(\rho)$ is defined on all of $J$, then we have $\|R(S_I)\| = \rho_s$. Otherwise, there exists an interval $(\overline{\rho}, 1]$ with $\rho_s < \overline{\rho} < 1$ such that $\nu(\rho) \to \infty$ as $\rho \to \overline{\rho}+$ and $\|R(S_I)\| = \overline{\rho}$.
\end{thm}

In the case at hand, $I$ is a finite Blaschke product and the operator we consider is finite rank so, consequently, $\rho_s = 0$. Thus we focus on solving for $\overline{\rho}$.
Foias and Tannenbaum give an explicit algorithm to find $\overline{\rho}$, which we recall here.

A computation shows that 
\[
\left( I-\frac{1}{4\rho^2}\left(I+aS_B+\overline a S_B^* + |a|^2 (S_BS_B^*-I) + |a|^2 I\right) \right) \nu_*(\rho)=\mu_*.
\]
But
\[
(I-S_BS_B^*)\nu_*(\rho)=\nu(\rho) \mu_*,
\]
so we have
\begin{equation}\label{eq:21}
\left(\left(
1- \frac{1}{4 \rho^2} - \frac{|a|^2}{4 \rho^2} \right)I - \frac{a}{4 \rho^2} S_B
- \frac{\overline a}{4 \rho^2} S_B^* 
\right) \nu_*(\rho)=
\left(1 - \frac{|a|^2}{4 \rho^2} \nu(\rho) \right) \mu_*.
\end{equation}
Since $\overline{B} \nu_*(\rho) \perp H^2$ we may write
$\overline B \nu_*(\rho)= \nu_{-1}\overline z + o(\overline z^2)$. Then
\begin{eqnarray*}
S_B^* \nu_*(\rho) &=& \overline z( \nu_*(\rho)-\nu_*(\rho)(0)) ,\\
S_B \nu_*(\rho) &=& z \nu_*(\rho) - B \nu_{-1},
\end{eqnarray*}
and note that $\nu_*(\rho)(0)=\nu(\rho)$.

Substituting these into \eqref{eq:21}, we get
\begin{eqnarray}\label{eq:3}
\left(
1- \frac{1}{4 \rho^2} - \frac{|a|^2}{4 \rho^2} - \frac{a}{4 \rho^2} z -   \frac{\overline a}{4 \rho^2} \overline z \right)
\nu_*(\rho)
&+& \frac{aB}{4 \rho^2} \nu_{-1}
+ \frac{\overline a}{4 \rho^2} \overline z \nu(\rho)
\\ &=& \left(1 - \frac{|a|^2}{4 \rho^2} \nu(\rho) \right) \mu_*.  \nonumber
\end{eqnarray}
Now we see that the coefficient of $\nu_*(\rho)$ is $0$ (on the unit circle) if
 \[
\frac{az^2}{4 \rho^2} - z \left( 1- \frac{1}{4\rho^2}-\frac{|a|^2}{4 \rho^2} \right)
+ \frac {\overline a}{4 \rho^2} =0. 
\]
Let the roots be $z_1, z_2$, so that
  $z_1 z_2 =\frac{\overline a}{a}$ and 
\begin{equation}\label{eq:middleterms}
z_1+z_2 = 
 \frac{4\rho^2-1-|a|^2}{a}.
\end{equation}
Upon substituting these zeros into \eqref{eq:3}, we then have for $k = 1, 2$
\[
\frac{az_kB(z_k)}{4 \rho^2}  \nu_{-1} + \left(\frac{\overline a}{4 \rho^2}+ \frac{|a|^2}{4 \rho^2} z_k \mu_*(z_k) \right) \nu(\rho) 
=  z_k\mu_*(z_k).
\]

We have two unknowns, $\nu_{-1}$ and $\nu(\rho)$, and solving for the latter from these
two simultaneous equations gives
\begin{eqnarray*}
\nu(\rho)  \big( (\overline a + |a|^2 z_1 \mu_*(z_1))z_2B(z_2)
- (\overline a + |a|^2 z_2 \mu_*(z_2))z_1B(z_1)\big)\\
=  4\rho^2  (z_1\mu_*(z_1)z_2B(z_2)-
z_2 \mu_*(z_2) z_1 B(z_1)).
\end{eqnarray*}
We may now consider the case that 
$P_\rho$ is singular (and $\overline\rho =\rho > \rho_s$), which happens when
\[
\overline a(z_2B(z_2)-z_1B(z_1))+|a|^2 z_1z_2 (\mu_*(z_1)B(z_2)-\mu_*(z_2)B(z_1))=0.
\]
On using the facts that $z_1z_2=\overline a/a$ and $\mu_*(z)=1-B(z)\overline{B(0)}$, this simplifies to
\begin{equation}\label{eq:zBz}
z_2B(z_2)-z_1B(z_1) + \overline a (B(z_2)-B(z_1))=0.
\end{equation}

This is the key equation for determining $z_1$ and $z_2$, but
in the special case that  
$a$ is real 
and the zeros of $B$ are also real, we 
can go further, since we
have
$B(\overline z)=\overline{B(z)}$ and \eqref{eq:zBz} is simply
\[
\imag (z_1 B(z_1)) +  a \imag  B(z_1) = 0.
\]

As $a \to 0$, the norm of $(I + a S_B)/2 \to 1/2$, so we are interested in values of $\rho$ near $1/2$. Note that for $\rho = (1+\delta)/2$  
where $
\delta$ is small, and
$a$  real, we see from \eqref{eq:middleterms} that
\[
\re z_1=\re z_2 = \frac{(1+\delta)^2-1-a^2}{2 a}=\frac{2\delta + \delta^2-a^2}{2a}.
\]
The equation for $z_1= \overline{z_2}$ reduces to 
\[
\imag (z B(z))=
 -a \imag(B(z)),
\]
and as $a \to 0$, we have $\imag z B(z) \to 0$ so $z B(z) \to \pm 1$.\\

These formulae will be used extensively in the rest of this section.

\begin{remark}
The Foias--Tannenbaum approach can be used in principle to calculate the norm of more general
functions of $S_B$, namely, $R(S_B)$ for $R$ a rational function, rather than simply $R(z)=1+az$.
The details
are more complicated, but these ideas may shed some light on the Crouzeix conjecture.
\end{remark}

\subsection{Blaschke products with real zeros}
We now apply this to Blaschke products with distinct real zeros and end this section with several detailed examples. Theorem~\ref{thm:nrreal} allows us to obtain the numerical radius of $I + a S_B$ from the limit appearing in Theorem~\ref{thm:main}. 

\begin{thm}\label{thm:nrreal} Let $B$ be a Blaschke product with  distinct   zeros $z_1, \ldots, z_n \in \mathbb{R}$. Then $w(S_B)$ is attained on the real line. \label{thm:numradreal}
\end{thm}

\begin{proof} Suppose that the zeros, $z_1, \ldots, z_n$ with $n \ge 2$ are real and lie in the line segment $[a, b]$. Rotating the line segment, if necessary, we may assume that $|a| \le b$. From the discussion above we know that for $t > 0$
\[
\|I + t e^{i \theta} S_B\| = \|(1 + t e^{i\theta} z) + B H^\infty \|.
\]
We claim that it is enough to show that if we can solve the interpolation problem $f(z_k) = 1 + t z_k$, for small
$t > 0$,   and $\|f\|_\infty \le \gamma$, where 
\[
\gamma = \gamma(t) = 1 + ct + o(t),
\]
 then we can solve $h(z_k) = 1 + t e^{i \theta} z_k$ with $\|h\|_\infty \le \gamma + o(t)$. If this is the case then we know that for $e^{i \theta}$ with $0 < \theta < 2\pi$, 
 \[\|(1 + t e^{i \theta} z) + B H^\infty\| \le \|h + (1 + t e^{i \theta} z - h) + B H^\infty \| \le \|h\| \le \gamma + o(t).\] Since we take $t \to 0$, this together with Theorem~\ref{thm:main} yields the result.

We first note that $S_B$ can have no eigenvalues on the boundary of $W(S_B)$. 
Though this is well known, see \cite{DON}, in this special case, we can give a short independent proof: If $z_j$ is an arbitrary zero of $B$, we can reorder the zeros so that $z_j$ and $z_k$, with $j \ne k$, are the first two zeros. 
The upper-left $2 \times 2$ corner of the matrix representing $S_B$ (see \eqref{eqn:A}) is then upper triangular and the $(1, 2)$ entry is $m: = \sqrt{1 - |z_j|^2} \sqrt{1 - |z_k|^2} > 0$. 
The numerical range of this $2 \times 2$ block, which is an elliptical disk with foci at $z_j$ and $z_k$ and minor axis $m$, is contained in $W(S_B)$. 
Since the minor axis has nonzero length, the elliptical disk is nondegenerate and the foci cannot lie on the boundary. Therefore $c > b$.

Choose an automorphism $\psi$ of $\gamma \mathbb{D}$ with $\psi(1) = 1$ that takes the point $1 + \lambda t$ to $1 + e^{i \theta} \lambda t + O(t^2)$ for $\lambda \in [a, b]$ and all small $t$. (In fact, the image of this interval lies on a circular arc that meets $\gamma \mathbb{T}$ twice at right angles.) By the chain rule, the derivative of $\psi$ at the fixed point is the unimodular constant $e^{i \theta}$. Therefore, 
\[\psi(z) = \psi(1) + \psi^\prime(1)(z - 1) + O(t^2) = 1 + e^{i \theta} (z-1) + O(t^2).\]
Thus, for $\lambda \in [a, b]$,
\[\psi(1 + \lambda t) = 1 + e^{i \theta} \lambda t + O(t^2).\]

Now consider $h := \psi \circ f$. Then $\|h\| \le \gamma + o(t)$ and $h(1 + t z_k) = 1 + e^{i \theta} t z_k + O(t^2)$. By adding on another analytic function of norm $O(t^2)$, where the bound depends on $n$ but not on $t$, we can solve the exact interpolation problem. 
\end{proof}

If the zeros of $B$ are not distinct, we may approximate $B$ uniformly by Blaschke products $B_n$ of the same degree with distinct zeros. In this case, the numerical radii of $S_{B_n}$ converge to the numerical radius of $S_B$ and we obtain the result for general finite Blaschke products.  

\subsection{Degree-$2$ Blaschke products}

We may now use Theorem~\ref{thm:numradreal} and the algorithm presented in Section~\ref{sec:norm} to compute the numerical radius of compressions of the shift operator with Blaschke products with real zeros. We begin with a simple example. 

\begin{exam}
If $B$ is a degree-$2$ Blaschke product with real zeros, $a_1$ and $a_2$, then the numerical radius of $S_B$ is
\begin{equation}\label{eqn:deg2}
\max \left\{\left|\frac{a_1 + a_2 - a_1 a_2 + 1}{2}\right|, \left|\frac{a_1 + a_2 + a_1 a_2 - 1}{2}\right|\right\}.
\end{equation} 
\end{exam}

We are looking for the points at which $z B(z) = \pm 1$. In general,
\[z^3 - (a_1 + a_2) z^2 + a_1 a_2 z = \pm (1 - (a_1+ a_2)z + a_1 a_2 z^2).\]
If the zeros are real and we let $z_1, z_2, z_3$ denote the three solutions to this equation, in one case we have
\[z_1 = -1~\mbox{and}~\re z_2 = \re z_3 = \frac{a_1 + a_2 + 1 - a_1 a_2}{2}\] and in the other case we have
\[z_1 = 1~\mbox{and}~\re z_2 = \re z_3 = \frac{a_1 + a_2 - 1 + a_1 a_2}{2},\] which yields \eqref{eqn:deg2}. 

We can check this result: In the event that the zeros of $B$ are both real, it follows from the elliptical range theorem that the numerical range is an elliptical disk with foci at $a_1$ and $a_2$ and major axis of length $|1 - \overline{a_1} a_2| = |1 - a_1 a_2|$, in this case. Thus, we obtain the correct result.\\

In particular, consider
\[
B(z) = z \left( \frac{z- \frac12}{1-\frac{z}2}\right).
\]
We find the three solutions to $z B(z) = 1$ are $-1, 3/4 - i \sqrt{7}/4$ and $3/4 + i \sqrt{7}/4$, and the solutions to $z B(z) = -1$ are $1, -1/4 + i \sqrt{15}/4, -1/4 - i \sqrt{15}/4$. 
Thus we see that $w(S_B) = 3/4$, and indeed one may verify that
\[
\|1+t S_B\|=1+\frac34 t + o(t),
\]
  as  $t \to 0+$. In fact, we know that in this case the numerical radius of $S_B$  is the maximum real value of the ellipse with foci $0$ and $1/2$ and major axis of length $1$: this is the center of the ellipse plus the length of the semi-major axis, which is $3/4$.\\

Now, there is an alternative way of calculating $\|1+tS_B\|$, namely, to find the largest $\gamma$ that makes the Pick matrix singular, see \cite[Section 2.1]{P97}.
Again we take $z_1=0$, $z_2=1/2$, for which we have seen that the numerical radius is $3/4$. Take $t>0$.

The Pick matrix, as given in \eqref{eq:pickmatrix}, is
\[
\begin{pmatrix}
1-1/\gamma^2 &  1-(1+t/2)/\gamma^2 \\
1-(1+t/2)/\gamma^2 & \frac{1-(1+t/2)^2/\gamma^2}{3/4}
\end{pmatrix}
\]
and is singular if
\[
\frac43(\gamma^2-1)(\gamma^2-x^2)-(\gamma^2-x)^2=0,
\]
where $x=1+t/2$.
This gives
\[
\gamma^4+\gamma^2(-4+6x-4x^2)+x^2=0.
\]
So
\[
\gamma^2= \frac{2+t+t^2 \pm \sqrt{(2+t+t^2)^2-4(1+t+t^2/4)}}{2} = (1+ t/2 ) \pm t + O(t^2),
\]
and $\|1+tS_B\|= 1+ 3t/4 + O(t^2)$, as expected.

\begin{exam}
A second simple example is the Blaschke product $B(z)=z^n$. In this case, $S_B$ is represented by a nilpotent Jordan block of size $n$ with zeros on the diagonal. The numerical
radius is $\cos (\pi/(n+1))$, see for example, \cite{HH92}.\end{exam}

To show this without referring to known results, we  solve $z B(z) = z^{n+1} = \pm 1$ and take the two values with the largest and smallest real parts. 
The cases of $n$ even and $n$ odd need to be considered separately, but
we obtain the two sets of points with real part $\cos(\pi/(n+1))$ and $-\cos(\pi/(n+1))$, establishing the result.  Since all zeros are real no matter how we rotate, we see that the numerical range is the closed disk of radius $\cos (\pi/(n+1))$.

\subsection{Degree-$3$ Blaschke products}\label{sec:deg3}

Very little is known about the numerical radius of $S_B$ for high degree Blaschke products, but for low degrees (namely degree $3$ and degree $4$) we can often compute $w(S_B)$ explicitly. 
 Gaaya \cite{HG2010} analyzed the numerical radius of compressions of the shift operator in the particular case that $B(z) = \left(\frac{z - a}{1 - \overline{a}z}\right)^n$. Our techniques also apply to such Blaschke products, since we may assume that the zero $a$ is real. In this section, we analyze the numerical radius of all $S_B$ with $B$ degree $3$ having real zeros.

\begin{exam} \label{ex:deg3}
We now compute the numerical radius of $S_B$ when $B$ is a Blaschke product of degree $3$ with real zeros $a$, $b$, and $c$. In fact, letting $\alpha = a + b + c + abc$ and $\beta = ab + ac + bc$, we see that $w(S_B)$ is \[ \max\left\{\left|\frac{\alpha + \sqrt{\alpha^2 - 8(\beta -1)}}{4}\right|,
\left|\frac{\alpha - \sqrt{\alpha^2 - 8(\beta - 1)}}{4}\right|\right\}.\]
\end{exam}

\begin{proof} This can be handled in the same way as the degree-$2$ case.  First we solve $z B(z) = \pm 1$. The equation we need is $z B(z) = -1$, because the other has $1$ and $-1$ as a root and therefore the tangent lines do not give the maximum real part. So we solve
\begin{equation}\label{eqn:deg3}
z(z-a)(z-b)(z-c) +(1-az)(1-bz)(1-cz) = 0.
\end{equation}
The left-hand side is
\begin{multline*} 1 - (a + b + c + abc)z +2(ab+ac+bc)z^2\\-(a+b+c+abc)z^3+z^4 = 1 - \alpha z + 2 \beta z^2 - \alpha z^3 + z^4.
\end{multline*} 
Since all coefficients are real the roots must occur in conjugate pairs, the roots must be of the form $x \pm i y$ and $u \pm iv$ with $x, y, u, v$ real and $x^2 + y^2 = u^2 + v^2 = 1$.  Thus,  \eqref{eqn:deg3} becomes
\[
     (z^2-2xz+(x^2+y^2))(z^2-2uz+(u^2+v^2)) = 0,
     \] 
     and therefore
     \[2(u+x) = a + b + c + abc \qquad\mbox{and}\qquad ab+ac+bc = 1 + 2ux.\]
So,
\[2(ux+x^2) = (a+b+c+abc)x\] and thus
\begin{equation}\label{eqn:8.5}
2x^2 - (a+b+c+abc)x +(ab+ac+bc-1) = 0.\end{equation} Now we can solve the general problem for a degree-$3$ Blaschke product with all zeros real:
\begin{equation}\label{eqn:realpart}
x = \frac{(a + b + c + abc) \pm \sqrt{(a+b+c+abc)^2 - 8(ab+ac+bc-1)}}{4}.
\end{equation}
The maximum modulus of the two values of $x$ is the numerical radius of $S_B$.\end{proof}

We work one particular case in detail. Let $B$ have real zeros at $0$, $a$ and $b$. Then our computations require us to solve $z B(z) = \pm 1$. From \eqref{eqn:realpart} we see that the real parts of the roots satisfy 
     \[
     r^2-\frac{(a+b)r}{2}+\frac{ab-1}{2}=0,
     \] or
     \[
     r= \frac{a+b \pm \sqrt{(a+b)^2-8(ab-1) }}{4}.
     \]
 By Theorem~\ref{thm:nrreal} the numerical range is attained on the real axis. Using the algorithm from Section~\ref{sec:FT}, we know the maximum will occur at
 \begin{equation}
 \label{eqn:8.5}
 \max\left\{\left|\frac{a+b + \sqrt{(a+b)^2-8(ab-1) }}{4}\right|, \left|\frac{a+b - \sqrt{(a+b)^2-8(ab-1) }}{4}\right|\right\}.
 \end{equation}
 We can check that this is correct in particular cases, one of which we do below.

If $a = b$, we would like to check that the numerical radius of $S_B$ \begin{equation}\label{eqn:ma}
w(S_B) = \max\left\{\left|\frac{a}{2} + \frac{\sqrt{2 - a^2}}{2}\right|, \left|\frac{a}{2} - \frac{\sqrt{2 - a^2}}{2}\right|\right\},\end{equation}
where 
\[B(z) = z \left(\frac{z - a}{1 - \overline{a} z}\right)^2.\] 

In this case, \[z B(z) = z^2 \left(\frac{z-a}{1 - \overline{a} z}\right)^2\] is a composition of two degree-2 Blaschke products. By \cite[Theorem 3.6]{GW2016} (see also \cite{F}) $W(S_B)$ is an elliptical disk and the foci are chosen from among the zeros of $B$. Writing $zB(z) = C(D(z))$ with $C(0) = D(0) = 0$, the zero of $D(z)/z$ is the one that is not a focus. Therefore, we see that of the three zeros $0, a$, and $a$ of $B$, the foci must be $0$ and $a$. 
From  \eqref{eqn:A} the matrix representation for $S_B$ as

\[\begin{bmatrix}
a & \sqrt{1 - |a|^2}\sqrt{1 - |b|^2} & -\overline{b}\sqrt{1 - |a|^2}\sqrt{1 - |c|^2}\\
0 & b & \sqrt{1 - |b|^2}\sqrt{1 - |c|^2}\\
0 & 0 & c
\end{bmatrix}\]
with $a = b$ and $c = 0$. It follows from the main theorem of \cite{KRS1997}  that the length of the minor axis is
\[\left(tr(A^\star A) - |a|^2 - |b|^2 - |c|^2\right)^{1/2},\] which is $\left(2(1 - a^2)\right)^{1/2}$ in this case.
So the length of the minor axis is $\sqrt{2(1 - a^2)}$. The center of the ellipse is $a/2$ and the foci are $0$ and $a$, so the major axis has length $\sqrt{2 - a^2}$. This agrees with \eqref{eqn:ma}.



\subsection{Degree-$4$ Blaschke products}

There is some work in \cite{Gau06}  on numerical ranges of $4 \times 4$ matrices, although
only in the case of an elliptical disk. The computations are extremely complicated, but
here we may take a more transparent approach.

We consider a Blaschke product with real zeros $a, b, c, d$, where things are necessarily more complicated. Once again, we must solve
\[z (-a + z) (-b + z) (-c + z) (-d + z) \pm  (1 - a z) (1 - b z) (1 - 
    c z) (1 - d z) = 0.\]
For the case    
\[z (-a + z) (-b + z) (-c + z) (-d + z) -  (1 - a z) (1 - b z) (1 - 
    c z) (1 - d z) = 0,\]    
a computation shows that this is equivalent to   
\begin{multline*}(-1 + z) (1 - (-1 + a + b + c + d + a b c d) z + ((-1 + c) (-1 + d) + 
      b (-1 + c + d + c d) +\\ a (-1 + c + d + c d + b (1 + c + d - c d))) z^2 - (-1 + a + b + 
      c + d + a b c d) z^3 + z^4\huge) = 0.\end{multline*} 
 For
 \[z (-a + z) (-b + z) (-c + z) (-d + z) + (1 - a z) (1 - b z) (1 - 
    c z) (1 - d z) = 0,\]
 a computation shows that this is equivalent to
\begin{multline*}(1 + z) (1 - (1 + b + c + d - a (-1 + b c d)) z + ((1 + c) (1 + d) + 
      b (1 + c + d - c d) + \\
      a (1 + c + d - c d - b (-1 + c + d + c d))) z^2 - (1 + b + c + 
      d - a (-1 + b c d)) z^3 + z^4) = 0.\end{multline*}
          
Noting that the techniques used in the solution in Section~\ref{sec:deg3} can be applied here, we see that if we let $\alpha$ denote the negative of the coefficient of $z$ and $\beta$ half the coefficient of $z^2$ we obtain:
\[(1 + z) (1 -\alpha z +  2 \beta z^2 - \alpha z^3 + z^4) = 0.\] By the techniques in the previous section
\begin{equation}\label{eqn:realpart2}
x = \frac{\alpha \pm \sqrt{\alpha^2 - 8(\beta - 1)}}{4}.
\end{equation}
  
There is also a standard procedure to solve a quartic (see \url{http://www.sosmath.com/algebra/factor/fac12/fac12.html} for a complete description) and in this case we obtain exact solutions using this process (or Mathematica).




\section{Norms of Truncated Toeplitz Operators}
\label{TTO}

In this section, we look at more general truncated Toeplitz operators.

Recall that for $u$ an inner function $P_u$ denotes the orthogonal projection of $L^2$ onto $K_u$. For $\varphi \in L^2$ the truncated Toeplitz operator 
\[A_\varphi^u f := P_u(\varphi f), \qquad f \in K_u\] is densely defined on $K_u^\infty : = H^\infty \cap K_u$. When $\varphi(z) = z$, we have the compression of the shift operator, $S_u$.

 In \cite[Corollary 2]{GR2010}, the authors provide a general lower bound for $\|A_\varphi^u\|$ for $\varphi \in L^2$ and obtain the following as a corollary: If $u$ is an inner function with zeros accumulating at every point of $\mathbb{T}$ and $\varphi$ is a continuous function on $\mathbb{T}$, then $\|A_\varphi^u\| = \|\varphi\|_\infty$. The authors note that the hypothesis can be weakened to the following: 
 
 \begin{prop}\label{GR} Let $u$ be an inner function that is not a finite Blaschke product. Let $\xi$ be a limit point of the zeros of $u$.  If $\varphi \in L^\infty$ is continuous on an open arc containing $\xi$ with $|\varphi(\xi)| = \|\varphi\|_\infty$, then $\|A_\varphi^u\| = \|\varphi\|$. \end{prop}
 
 We present a short proof of their general result in Proposition~\ref{GR2} and show how these same techniques allow the result to be generalized.
 
Recall that $C(\mathbb{T})$ denotes the algebra of continuous functions on the unit circle. We let $M(H^\infty)$ denote the maximal ideal space of $H^\infty$, or set of nonzero multiplicative linear functionals with the weak-$\star$ topology. Then, identifying a point $z \in \mathbb{D}$ with the linear functional that is evaluation at that  point, we may think of $\mathbb{D}$ as contained in $M(H^\infty)$. Carleson's corona theorem says that $\mathbb{D}$ is dense in $M(H^\infty)$. Let \[M(H^\infty + C(\mathbb{T})) = M(H^\infty) \setminus \mathbb{D}\] be the maximal ideal space of the (closed) algebra $H^\infty + C(\mathbb{T})$.  

Let $Z(u)$ denote the zeros of an inner function $u$ in $M(H^\infty)$ and $Z_\mathbb{D}(u)$ the zeros of $u$ in $\mathbb{D}$. If the function $f \notin H^\infty$, we write $f(z)$ for the Poisson extension of $f$ evaluated at a point $z$. Note that the proposition below does not require $u$ to have zeros in $\mathbb{D}$ and can thus be applied readily to inner functions with a nontrivial singular inner factor.

 If $u$ is a bounded harmonic function and $v$ the harmonic conjugate of $u$ and we let $h = e^{u + iv}$, then $u$ extends to a continuous function on $M(H^\infty)$ defined by $\hat{u} = \log|\hat{h}|$, see \cite[Lemma 4.4]{HOFFMAN}. Thus, for $f \in L^\infty$ we see that $\hat{f}$ is continuous on $M(H^\infty)$. There is another way to look at this: Recall that the maximal ideal space of $L^\infty$ is the Shilov boundary of $M(H^\infty)$. Then for $f \in L^\infty$ and $x \in M(H^\infty)$ we have
 \[x(f) = \hat{f}(x) = \int_{\supp\mu_x} \hat{f} d\mu_x,\] where $\hat{f}$ denotes the Gelfand transform of $f$ and $\supp \mu_x$ denotes the subset of $M(L^\infty)$ that is the {\it support set} for the representing measure $\mu_x$ of $x$. It is common to write $f$ in place of $\hat{f}$. In this way, we may think of the Gelfand transform of $f$ as a continuous function on $M(H^\infty)$, see \cite[p. 184]{HOFFMANBOOK}. The support set of $x \in M(H^\infty)$ is known to be a weak peak set for $H^\infty$ (see \cite[p. 207]{HOFFMANBOOK}).

\begin{prop}\label{GR2} Let $u$ be an inner function not invertible in $H^\infty + C(\mathbb{T})$. For $f \in L^\infty$, if $\hat{f}(x) = \|f\|_\infty$ for some $x \in Z(u)$, then
\[\dist(f, uH^\infty) = \|f\|_\infty.\]
\end{prop}

\begin{proof}  
If $x \in \mathbb{D}$, then $x$ is a zero of $u$ and it is clear that 
\[\dist(f, uH^\infty) \ge |f(x)| = \|f\|_\infty.\] Since $\dist(f, uH^\infty) \le \|f\|_\infty$ the result holds. So we may suppose that $x \in M(H^\infty) \setminus \mathbb{D}$ and $x(f) = \|f\|_\infty$.  Since $d\mu_x$ is a probability measure, $f$ must be constant on the support set of $x$ and that constant must be $\|f\|_\infty$. Therefore,

\[\dist(f, uH^\infty) \ge |f(x)| = \|f\|_\infty,\]
which completes the proof. 
\end{proof}

 Note that the assumption in Proposition~\ref{GR} implies the assumption in Proposition~\ref{GR2}: In Proposition ~\ref{GR} we have $\varphi \in L^\infty$ continuous on an open arc about $\xi$ for which $|\varphi(\xi)| = \|\varphi\|_\infty$ and $\xi$ is a limit point of the zeros of $u$, so we may choose $x \in M(H^\infty + C)$ in the closure of the zeros of $u$ with $\hat{\varphi}(x) = \|\varphi\|_\infty$.

To see that Proposition 4.2 extends Proposition 4.1, we use Corollary 1 of Garcia and Ross's paper \cite{GR2010} to note that if $\varphi \in L^2$, then 
\[\sup_{\{\lambda \in \mathbb{D}: u(\lambda) = 0\}} |\hat\varphi(\lambda)| \le \|A_\varphi^u\|.\] Therefore, if $\varphi \in L^\infty$ and $x \in M(H^\infty + C)$ with $x$ in the closure of the zeros of $u$, and $|\hat{\varphi}(x)| = \|\varphi\|_\infty$, then
\[\|\varphi\|_\infty = |\hat{\varphi}(x)| = \sup_{\{\lambda \in \mathbb{D}: u(\lambda) = 0\}} |\hat\varphi(\lambda)| \le \|A_\varphi^u\| \le \|\varphi\|_\infty.\]


We now consider the so-called thin interpolating Blaschke products (defined below). While one can follow the procedure below to obtain estimates for other interpolating Blasch\-ke products, the estimates will not be as good. In any event, this gives us information about the essential norm of a Hankel operator.

Recall that a Blaschke product $B$ is {\it interpolating} if the zero sequence $(z_n)$ of $B$ is an interpolating sequence for $H^\infty$; that is, given a bounded sequence of complex numbers, $(w_n)$, there exists $f \in H^\infty$ with $f(z_n) = w_n$ for all $n$. Carleson showed that this is equivalent to the existence of $\delta > 0$ such that
\[\inf_n\prod_{m \ne n} \left|\frac{z_m - z_n}{1 - z_m \overline{z_n}}\right| \ge \delta.\]
Let \[\delta_n := \prod_{m \ne n} \left|\frac{z_m - z_n}{1 - z_m \overline{z_n}}\right|.\] If $\delta_n \to 1$ as $n \to \infty$, the interpolating sequence is said to be a {\it thin interpolating sequence}.  For example, a radial sequence $(z_n)$ for which \[(1 - |z_{n+1}|)/(1-|z_n|) \to 0\] as $n \to \infty$ is such a sequence. We will apply Earl's theorem (Theorem~\ref{thm:Earl}) below to Blaschke products for which the zero sequence is a thin interpolating sequence, \cite[Theorem 2]{EARL1969}. We isolate Earl's theorem here for easy reference.

\begin{thm}[Earl's Theorem]\label{thm:Earl} Suppose that $(z_n)$ is an interpolating sequence with
\[\inf_n\prod_{m \ne n} \left|\frac{z_m - z_n}{1 - z_m \overline{z_n}}\right| \ge \delta > 0.\] If $(w_n)$ is any bounded sequence of complex numbers, and $M$ is an arbitrary number greater than
\[\frac{2 - \delta^2 + 2(1 - \delta^2)^{1/2}}{\delta^2} \sup_n |w_n|,\] then there exists an $\alpha \in \mathbb{R}$ and
a  Blaschke product $B$ such that
\[M e^{i \alpha} B(z_j) = w_j~\mbox{for}~j = 1, 2, \ldots.\]
\end{thm}

From Earl's theorem, we see that if we are given $(w_n)$, a bounded sequence of complex numbers, we can find $g \in H^\infty$ such that $g(z_n) = w_n$ for every $n$ and 
\[\|g\| \le \frac{2 - \delta^2 + 2(1 - \delta^2)^{1/2}}{\delta^2} \sup_{n} |w_n|.\]

A second result will be useful here as well.

\begin{thm}\label{AGGIS}\cite{AXLERGORKIN, GUILLORYIZUCHISARASON} Let $f \in H^\infty + C(\mathbb{T})$ and let $B$ be an interpolating Blaschke product with zeros $(z_n)$. Then $\overline{B} f \in H^\infty + C(\mathbb{T})$ if and only if $f(z_n) \to 0$.
\end{thm}

Via the Chang-Marshall theorem, this theorem implies a more general result for closed subalgebras of $L^\infty$ containing $H^\infty$ with $u$ an arbitrary function in $L^\infty$ of norm at most one. We state the version that we will need below, with a reference to the more general statement.

\begin{thm}\label{AGGIS1}(\cite{AXLERGORKIN}, Theorem 4.) Let $h \in H^\infty + C$ and let $u$ be a function in $H^\infty + C$ with $\|u\| \le 1$. If $h(1 - |u|) = 0$ on $M(H^\infty) \setminus \mathbb{D}$, then $h H^\infty[\overline{u}] \subseteq H^\infty + C$.
\end{thm}

We remind the reader that we identify a function in $L^\infty$ with its Gelfand transform. For $f \in H^\infty + C$, when we evaluate $f$ at a point $z \in \mathbb{D}$, we are evaluating the Poisson extension of the function. In order to state our results on the disk, we need the following technical lemma.

\begin{lem}\label{prop:limsup} Let $B$ be an interpolating Blaschke product with zero sequence $(z_n)$ and $f \in H^\infty + C$. Then
\[\max\{|f(x)|: x \in Z(B) \cap M(H^\infty + C)\} = \limsup |f(z_n)|.\]
\end{lem}

\begin{proof}
Using continuity, choose $x_0 \in M(H^\infty + C)$ to be a point at which $B(x_0) = 0$ and $|f(x_0)| = \max\{|f(x)|: x \in Z(B) \cap M(H^\infty + C)\}$. 
Let \[\mathcal{O} = \{y \in M(H^\infty): |f(y) - f(x_0)| < 1/n\} \bigcap \big(M(H^\infty) \setminus \{z: |z| \le 1 - 1/n\}\big).\] Since the Gelfand transform is continuous on $M(H^\infty)$, this is an open set. Since we assume $B$ is interpolating and $B(x_0) = 0$, the point $x_0$ is in the closure of the zeros of $B$, \cite[p. 83]{HOFFMAN}. Thus there exists $z_n \in \mathbb{D} \cap \mathcal{O}$. Therefore,
\[|f(x_0)| - 1/n \le |f(z_n)| \le |f(x_0)| + 1/n~\mbox{and}~|z_n| \ge 1 - 1/n.\]
Therefore $\limsup_n |f(z_n)| \ge |f(x_0)|$.\\

Suppose that $\limsup_n|f(z_n)| > |f(x_0)| + \alpha$ for some $\alpha > 0$. Then for all $\varepsilon < \alpha/2$ there exists $N(\varepsilon)$ such that $\sup_{n \ge N(\varepsilon)}|f(z_n)| \ge |f(x_0)| + \alpha - \varepsilon > |f(x_0)| + \alpha/2$. Thus, we find a sequence of points $(z_{n_{N(\varepsilon}})$ for which $B(z_{n_{N(\varepsilon}}) = 0$ and 
\[|f(z_{n_{N(\varepsilon}})|> |f(x_0)| + \alpha/2.\]
Choosing a point $x$ in the closure, we obtain $x \in M(H^\infty+C)$ with $|f(x)| > |f(x_0)|$, a contradiction.
\end{proof}

If $B$ is an interpolating Blaschke product with zeros $(z_n)$ we define
\[\tilde{\delta_n} := \inf_{k > n} \prod_{m > n, m \ne k} \left|\frac{z_m - z_k}{1 - z_m \overline{z_k}}\right|.\]
Note that $\tilde{\delta_n}$ is a bounded increasing sequence of real numbers and therefore this sequence converges. We let $\tilde{\delta}$ denote this limit.

\begin{thm} Let $B$ be an interpolating Blaschke product with zeros $(z_n)$ and let $f \in H^\infty + C(\mathbb{T})$. Then
\[\limsup |f(z_n)| \le \|f + B(H^\infty + C(\mathbb{T}))\| \le \frac{2 - \tilde{\delta}^2 + 2(1 - \tilde{\delta}^2)^{1/2}}{\tilde{\delta}^2}\limsup|f(z_n)|.\] \end{thm}

\begin{proof} For every $h \in H^\infty + C(\mathbb{T})$ and $x \in M(H^\infty + C(\mathbb{T}))$ with $x(B) = 0$, we have
\[\|f + B h\| \ge |x(f) + x(B h)| = |x(f)|.\] Thus, 
\[\|f + B(H^\infty + C(\mathbb{T}))\| \ge \max \{|f(x)|: x \in M(H^\infty + C(\mathbb{T})), x(B) = 0\}.\]
By Proposition~\ref{prop:limsup}, we obtain the lower inequality.

For the upper inequality, we note that the conjugate of the function $b_N(z):=\prod_{j = 1}^N \frac{z - z_j}{1 - \overline{z_j}z}$ lies in $H^\infty + C(\mathbb{T})$ and therefore, writing $B_N = B \overline{b_N}$ we see that $B_N$ is the Blaschke product with zero sequence $(z_n)_{n > N}$. In particular, $\delta(B_N) \ge \tilde{\delta}_N$. Now, $H^\infty + C(\mathbb{T})$ is an algebra and $1 = \overline{b_N} b_N$, so
\[\|f + B(H^\infty+C(\mathbb{T}))\| = \|f + (B\overline{b_N})(H^\infty + C(\mathbb{T}))\|= \|f + B_N (H^\infty+C(\mathbb{T}))\|.\] Choose $g_N \in H^\infty$, using Earl's theorem, so that $g_N = f$ on $(z_j)_{j > N}$ and $g_N$ satisfies the norm estimates \[\|g_N\| \le \frac{2 - \tilde{\delta}_N^2 + 2(1 - \tilde{\delta}_N^2)^{1/2}}{\delta_N^2} \sup_{n > N} |w_n|.\] Then $f - g_N \in B_N(H^\infty+C(\mathbb{T})) = B(H^\infty + C(\mathbb{T}))$, by Theorem~\ref{AGGIS}. Therefore,
\begin{multline}\label{eqn:upper}
\|f + B(H^\infty + C(\mathbb{T}))\| = \|g_N + B(H^\infty+C(\mathbb{T}))\| \\
 \le  \|g_N\| \le \frac{2 - \tilde{\delta}_N^2 + 2(1 - \tilde{\delta}_N^2)^{1/2}}{\tilde{\delta}_N^2} \sup_{n > N} |f(z_n)|.
\end{multline}
Now $(\tilde{\delta}_n)$ is an increasing sequence and therefore the constants appearing on the right-hand side of \eqref{eqn:upper} form a decreasing sequence. Also $(\sup_{n > N} |f(z_n)|)$ is a decreasing sequence. Taking the limit yields the desired upper bound.

\end{proof}

For thin interpolating sequences the result is especially nice. In this case, $\delta_n \to 1$ implies that $\tilde{\delta}_n \to 1$. Thus we have the following corollary.

\begin{cor}
Let $B$ be a thin interpolating Blaschke product with zero sequence $(z_n)$ and let $f \in H^\infty + C(\mathbb{T})$. Then
\[ \|f + B(H^\infty + C(\mathbb{T}))\| = \limsup|f(z_n)|.\]
\end{cor}

\begin{remark}
Note that in \eqref{eqn:upper} if $f(z_n) \to 0$, then $f \in B(H^\infty + C(\mathbb{T}))$, which  is the result of Theorem~\ref{AGGIS}.
\end{remark}

We recall a result of Bessonov  \cite[Prop. 2.1]{B2015}.

\begin{prop}\label{Bes1} Let $u$ be an inner function and $\varphi \in H^\infty + C(\mathbb{T})$. Then the truncated Toeplitz operator $A_\varphi^u: K_u \to K_u$ is compact if and only if $\varphi \in u(H^\infty + C(\mathbb{T}))$. \end{prop}

It follows from Theorem~\ref{AGGIS} that if $B$ is an interpolating Blaschke product with zero sequence $(z_n)$ and $f \in H^\infty + C(\mathbb{T})$, then 
$A_f^u$ is compact if and only if $f(z_n) \to 0$. \\

Ahern and Clark proved the following theorem (see also \cite{GRW}).

\begin{thm}\cite[Section 5]{AC70}\label{AC} Let $f$ be continuous and let $u$ be an inner function. The operator $A_f^u$ is compact if and only if $f(e^{i \theta}) = 0$ for all $e^{i \theta} \in \supp   u \cap \mathbb{T}$. \end{thm}

We provide a more general result below.

\begin{thm} Let $f \in H^\infty + C(\mathbb{T})$ and let $u$ be an inner function. The operator $A_f^{(u^n)}$ is compact for every $n \in \mathbb{N}$ if and only if
\[\lim_{|z| \to 1^-}|f(z)|(1 - |u(z)|) = 0.\] \end{thm}

\begin{proof} 
By our assumption and the corona theorem (which states that $\mathbb{D}$ is dense in $M(H^\infty)$), $|f(x)|(1 - |u(x)|) = 0$ on $M(H^\infty + C(\mathbb{T}))$. By Theorem~\ref{AGGIS1}, $f \overline{u^n} \in H^\infty + C(\mathbb{T})$ for every $n$. Thus, $f \in u^n(H^\infty + C(\mathbb{T}))$ and the result follows from Proposition~\ref{Bes1}.

For the other direction, suppose that $A_f^{(u^n)}$ is compact for every $n$. Then, by Proposition~\ref{Bes1}, $f \overline{u}^n \in H^\infty + C(\mathbb{T})$ for every $n$. If $x \in M(H^\infty + C(\mathbb{T}))$ with $|u(x)| \le r < 1$, we know that $f = u^n h_n$ and $\|h_n\|_\infty = \|f\|$. Therefore, since $x \in M(H^\infty + C(\mathbb{T}))$, we have $|x(f)| = |x(u)^n x(h_n)| \le |x(u)|^n \|f\|_\infty \le r^n$ for every $n$. 
Thus $x(f) = 0$ if $|u(x)| < 1$. Therefore $|f(x)| (1 - |u(x)|) = 0$ on $M(H^\infty + C(\mathbb{T}))$. Since $\mathbb{D}$ is dense in $H^\infty$ and $M(H^\infty + C(\mathbb{T})) = M(H^\infty) \setminus \mathbb{D}$ we have
\[\lim_{|z| \to 1^-} |f(z)| (1 - |u(z)|) = 0.\] \end{proof}

To get the Ahern and Clark result (Theorem~\ref{AC}), let $f$ be continuous. Then $f \overline{u} \in H^\infty + C(\mathbb{T})$ if and only if $f$ vanishes at each discontinuity of $u$. This happens if and only if $f$ vanishes at each discontinuity of $u^n$. And this, in turn, happens if and only if $f \bar{u}^n \in H^\infty + C(\mathbb{T})$ for all $n$. Thus, Ahern and Clark's result follows.

 In a second survey on truncated Toeplitz operators \cite[p. 12]{CFT}, Chalendar, Ross, and Timotin say that it would be interesting to give an example of a compact truncated Toeplitz operator with a symbol $\psi \in u(H^\infty + C(\mathbb{T}))$ that has no continuous symbol. We present such an example below.

\begin{exam}\label{ex:4.2}
Let $b$ be an interpolating Blaschke product with zero sequence $(z_n)$ clustering at every point of the unit circle. Then there exists $f \in H^\infty + C(\mathbb{T})$ with $f(z_n) \to 0$ and $f(z_m) \ne 0$ for some $m$ such that $A_f^b$ is compact and $A_f^b$ has no continuous symbol.

\end{exam}

\begin{proof} That such interpolating Blaschke products exist is well known (and, while they are easy to construct by wrapping around the circle  while increasing the modulus at a fast enough rate, it is also clear from the Chang--Marshall theorem that many of these exist). Choose $w_n \ne z_k$ for all $k$ such that $\rho(z_n, w_n) \to 0$ as $n \to \infty$. If $f$ is the corresponding Blaschke product with zeros at $(w_n)$, then $f(z_n) \to 0$ and $f(z_m) \ne 0$ by construction. By Theorem~\ref{AGGIS} and Proposition~\ref{Bes1}, $f \in b(H^\infty + C(\mathbb{T}))$ and $A_f^b$ is compact.

 If $A_f^b$ had a representation with continuous symbol $g$, we would have \cite[Theorem 3.1]{SAR2007} $f - g \in b H^2 + \overline{b} \overline{H^2}$. But $f(z_n) \to 0$, $b(z_n) = \overline{b(z_n)} = 0$ and therefore $g(z_n) \to 0$. But $\{z_n\}$ contains $\mathbb{T}$ in its closure and, since $g$ is continuous, we must have $g = 0$. Since $f(z_m) \ne 0$, this is impossible. \end{proof}
 
Remark: The proof shows that under these assumptions, if there is a continuous symbol $g$, then we must have $g b$ continuous.

 \subsection*{Truncated Hankel operators}

For an inner function $u$ and a suitable symbol $\phi$ the truncated Hankel operator $B^u_\phi: K_u \to \overline{zK_u}=\overline u K_u$
is defined by 
\[
B^u_{\phi}(f)=P_{\overline{zK_u}}(\phi f),
\]
and, as has been observed in \cite[Lem. 3.3]{B2015}, we have
$B^u_\phi(f) = \overline u A^u_{u\phi} f$ for $f \in K_u$. Since multiplication by $\overline u$ is
an isometry from $K_u$ onto $\overline u K_u$, the study
of compactness, boundedness and Schatten-class membership of truncated Hankel operators generally
reduces to that of truncated Toeplitz operators.

Thus Example  \ref{ex:4.2} provides a function $\phi \in (H^\infty+C(\TT)) \cap \overline b(H^\infty+C(\TT))$ with $B^b_\phi$ compact, but with no
symbol $\psi$ such that $b\psi$ is continuous.

\subsection*{Acknowledgements}

This work was partially supported by   grants from the Simons Foundation ($\sharp$243653 to Pamela Gorkin)
and the London Mathematical Society ($\sharp$41621 to Jonathan Partington).
Since August 2018, PG has been serving as a Program Director in the Division of Mathematical Sciences
at the National Science Foundation (NSF), USA, and as a component of this position, she
received support from NSF for research, which included work on this paper. Any opinions,
findings, and conclusions or recommendations expressed in this material are those of the
authors and do not necessarily reflect the views of the National Science Foundation.\\

The authors also wish to thank the referees for their helpful comments.

\noindent
\author{Pamela Gorkin\\ \small{Department of Mathematics, Bucknell University, Lewisburg, PA 17837, U.S.A.} \\ \small{\tt E-mail: pgorkin@bucknell.edu} \and and \\  \\Jonathan R. Partington\\\small{School of Mathematics, University of Leeds, Leeds LS2 9JT, U.K.} \\ \small{\tt E-mail: J.R.Partington@leeds.ac.uk}}


\begin{thebibliography}{99}

\bibitem{AC70} Ahern, P. R. and Clark, D. N., On functions orthogonal to invariant subspaces. {\it Acta Math.} 124 (1970) 191--204.

\bibitem{AM}
Allwright, D. J. and Mees, A. I.,
Stability, extended spaces and numerical ranges.
{\em SIAM J. Control Optim.} 20 (1982), no. 3, 328--337. 

\bibitem{AXLERGORKIN} Axler, Sheldon and Gorkin, Pamela, Divisibility in Douglas algebras. {\em Michigan Math. J.} 31 (1984), no. 1, 89--94.

\bibitem{B2015} Bessonov, R. V., Fredholmness and compactness of truncated Toeplitz and Hankel operators. {\it Integral Equations Operator Theory} 82 (2015), no. 4, 451--467.

\bibitem{BD71}
Bonsall, F.F. and Duncan, J., {\em Numerical ranges of operators on normed spaces and of elements of normed algebras}. London Mathematical Society Lecture Note Series, 2 Cambridge University Press, London-New York, 1971.

 \bibitem{CP16}
  C\^amara, M. Cristina and Partington, Jonathan R., Spectral properties of truncated Toeplitz operators by equivalence after extension. 
  {\em J. Math. Anal. Appl.} 433 (2016), no. 2, 762--784. 

\bibitem{CFT}
 Chalendar, Isabelle, Fricain, Emmanuel and Timotin, Dan, A survey of some recent results on truncated Toeplitz operators. {\em Recent progress on operator theory and approximation in spaces of analytic functions}, 59--77, Contemp. Math., 679, Amer. Math. Soc., Providence, RI, 2016.
 
\bibitem{C2004}
Crouzeix, Michel, Bounds for analytical functions of matrices. 
{\em Integral Equations Operator Theory\/} 48 (2004), no. 4, 461--477.

\bibitem{CP2017}
Crouzeix, M.; Palencia, C.,
The numerical range is a ($1+\sqrt{2}$)-spectral set.
{\it SIAM J. Matrix Anal. Appl.} 38 (2017), no. 2, 649--655.
  
  \bibitem{G2} Daepp, Ulrich; Gorkin, Pamela; Shaffer, Andrew and Voss, Karl, M\"obius transformations and Blaschke products: the geometric connection. {\it Linear Algebra Appl.} 516 (2017), 186--211.
 
 \bibitem{DON}
Donoghue, William F., Jr., On the numerical range of a bounded operator. 
{\em Michigan Math. J.}, 4 (1957) 261--263.

\bibitem{EARL1969}
Earl, J. P., On the interpolation of bounded sequences by bounded functions. {\em J. London Math. Soc.} (2) 2 (1970) 544--548.

\bibitem{fan-tits}
Fan, Michael K. H. and Tits, Andr\'e L.,
$m$-form numerical range and the computation of the structured singular value.
{\em IEEE Trans. Automat. Control\/} 33 (1988), no. 3, 284--289. 

\bibitem{FT87}
Foias, C. and Tannenbaum, A.,   On the Nehari problem for a certain class of $L^\infty$-functions appearing in control theory. {\em J. Funct. Anal.} 74 (1987), no. 1, 146--159. 

\bibitem{FTZ88}
Foias C., Tannenbaum A., and Zames, G.,
Some explicit formulae for the singular values of certain
              {H}ankel operators with factorizable symbol.
{\em SIAM J. Math. Anal.}
19 (1988), no. 5, 1081--1089.

\bibitem{F}
Fujimura, Masayo, Inscribed ellipses and Blaschke products. {\em Comput. Methods Funct. Theory\/} 13 (2013), no. 4, 557--573.

\bibitem{HG2010} Gaaya, Haykel On the numerical radius of the truncated adjoint shift. {\em Extracta Math}. 25 (2010), no. 2, 165--182.

\bibitem{GR2010}
Garcia, Stephan Ramon, and Ross, William T.,
The norm of a truncated Toeplitz operator. {\em Hilbert spaces of analytic functions}, 59 -- 64, 
CRM Proc. Lecture Notes, 51, Amer. Math. Soc., Providence, RI, 2010.

\bibitem{GR2013}
Garcia, Stephan Ramon and Ross, William T., Recent progress on truncated Toeplitz operators. Blaschke products and their applications, 275 -- 319, {\em Fields Inst. Commun.}, 65, Springer, New York, 2013.

\bibitem{GRW}
Garcia, Stephan Ramon, Ross, William T., and Wogen, Warren R., C*-algebras generated by truncated Toeplitz operators. 
{\em Concrete operators, spectral theory, operators in harmonic analysis and approximation}, 
181--192, Oper. Theory Adv. Appl., 236, Birkh\"auser/Springer, Basel, 2014.

\bibitem{GARNETT}
Garnett, John B., {\em Bounded analytic functions}. Pure and Applied Mathematics, 96.  Academic Press, Inc. [Harcourt Brace Jovanovich, Publishers], New York--London, 1981.




\bibitem{Gau06}
Gau, Hwa-Long, Elliptic numerical ranges of $4\times 4$ matrices. {\em Taiwanese J. Math.} 10 (2006), no. 1, 117--128.

\bibitem{GW}
Gau, Hwa-Long and Wu, Pei Yuan,  Numerical range of S($\phi$). {\em Linear and Multilinear Algebra\/} 45 (1998), no. 1, 49--73.

\bibitem{GW1}
Gau, Hwa-Long and Wu, Pei Yuan,  Numerical range and Poncelet property. {\em Taiwanese J. Math.} 7 (2003), no. 2, 173--193.

\bibitem{GW2}
Gau, Hwa-Long and Wu, Pei Yuan,  Numerical range circumscribed by two polygons. {\em Linear Algebra Appl.} 382 (2004), 155--170.






\bibitem{GW2016}
Gorkin, Pamela and Wagner, Nathan,
Ellipses and compositions of finite Blaschke products.
{\em J. Math. Anal. Appl.} 445 (2017), no. 2, 1354--1366. 

\bibitem{GUILLORYIZUCHISARASON} Guillory, Carroll, Izuchi, Keiji, and Sarason, Donald,
Interpolating Blaschke products and division in Douglas algebras. 
{\it Proc. Roy. Irish Acad. Sect.} A 84 (1984), no. 1, 1--7.

\bibitem{GR}
Gustafson, Karl E. and Rao, Duggirala K. M., {\it Numerical range. The field of values of linear operators and matrices}. Universitext. Springer-Verlag, New York, 1997.

\bibitem{HH92}
Haagerup U. and de la Harpe,  P.,
The numerical radius of a nilpotent operator on a Hilbert space. 
{\em Proc. Amer. Math. Soc.} 115 (1992), no. 2, 371--379. 


\bibitem{HOFFMAN}
Hoffman, Kenneth, Bounded analytic functions and Gleason parts. {\it Ann. of Math.} (2) 86 (1967), 74--111.

\bibitem{HOFFMANBOOK}
Hoffman, Kenneth, Banach spaces of analytic functions. Reprint of the 1962 original. {\it Dover Publications}, Inc., New York, 1988.


\bibitem{KRS1997}
Keeler, Dennis S., Rodman, Leiba, and Spitkovsky, Ilya M.,
The numerical range of $3 \times 3$ matrices.
{\it Linear Algebra Appl.} 252 (1997), 115--139.

\bibitem{KIPP} Kippenhahn, Rudolf, On the numerical range of a matrix. Translated from the German by Paul F. Zachlin and Michiel E. Hochstenbach. {\em Linear Multilinear Algebra} 56 (2008), no. 1--2, 185--225.

\bibitem{CKLI} Li, Chi-Kwong,
A simple proof of the elliptical range theorem. 
{\it Proc. Amer. Math. Soc.} 124 (1996), no. 7, 1985--1986.

\bibitem{M} Mirman, Boris, UB-matrices and conditions for Poncelet polygon to be closed. {\it Linear Algebra Appl.} 360 (2003), 123--150. 

\bibitem{P97} Partington, Jonathan R., {\em Interpolation, identification, and sampling}. London Mathematical Society Monographs. New Series, 17. The Clarendon Press, Oxford University Press, New York, 1997.




\bibitem{sarason67}
Sarason,  D., Generalized interpolation in $H^\infty$. {\em Trans. Amer. Math. Soc.} 127 (1967), 179--203.

\bibitem{SAR2007} Sarason, D., Algebraic properties of truncated Toeplitz operators, {\it Oper. Matrices\/} 1 (2007), no. 4, 491--526.

\bibitem{NF} Sz.-Nagy, B\'{e}la, Foias, Ciprian, Bercovici, Hari and K\'erchy, L\'aszl\'o, 
{\em Harmonic analysis of operators on Hilbert space}. Second edition. Revised and enlarged edition. Universitext. Springer, New York, 2010.


\end{thebibliography}
\end{document}